\documentclass[11pt]{article}\UseRawInputEncoding
\usepackage{amssymb,amsfonts,amsmath,latexsym,epsf,tikz,url}
\usepackage[usenames,dvipsnames]{pstricks}
\usepackage{pstricks-add}
\usepackage{epsfig}
\usepackage{pst-grad} % For gradients
\usepackage{pst-plot} % For axes
\usepackage[space]{grffile} % For spaces in paths
\usepackage{etoolbox} % For spaces in paths
\makeatletter % For spaces in paths
\patchcmd\Gread@eps{\@inputcheck#1 }{\@inputcheck"#1"\relax}{}{}
\makeatother

\newtheorem{theorem}{Theorem}[section]
\newtheorem{proposition}[theorem]{Proposition}

\newtheorem{corollary}[theorem]{Corollary}

\newcommand{\qed}{\hfill $\square$\medskip}

\begin{document}

\title{Sombor-index-like invariants  of some graphs}

\author{
Nima Ghanbari$^{}$\footnote{Corresponding author} \and Saeid Alikhani
}

\date{\today}

\maketitle

\begin{center}
	Department of Informatics, University of Bergen, P.O. Box 7803, 5020 Bergen, Norway\\
Department of Mathematical Sciences, Yazd University, 89195-741, Yazd, Iran\\
{\tt  Nima.ghanbari@uib.no, alikhani@yazd.ac.ir }
\end{center}

%%%%%%%%%%%%%%ABSTRACT%%%%%%%%%%%%%%%%%%%%%%%%%%%%%%%%%%%%%%%%%%%%%%%%%%%%%%%%%%%%

\begin{abstract}
 The Sombor index (SO) is a vertex-degree-based graph invariant, defined as the sum over all pairs
 of adjacent vertices of $\sqrt{d_i^2+d_j^2}$, where $d_i$ is the degree of the $i$-th vertex. It has been conceived using geometric
 considerations. Recently, a series of new SO-like degree-based graph invariants (denoted by $SO_1, SO_2,..., SO_6$) is taken into consideration,  when  the geometric background of several classical topological indices (Zagreb,
 Albertson) has considered. In this paper, we compute and study these new indices for some graphs, cactus chains and polymers. 
\end{abstract}


\medskip
\noindent{\bf AMS Subj.\ Class.:} 05C07; 05C09.

%%%%%%%%%%%%%%%%%%%%%%%%%%%%%%%%%%%%%%%%%%%%%%%%%%%%%%%%%%%%%%%%%%%%%%%%%%%%%%%%%
%%%%%%%%%%%%%%%%%%%%%%%%%%%%%%%%%%%%%%%%%%%%%%%%%%%%%%%%%%%%%%%%%%%%%%%%%%%%%%%%%
\section{Introduction}
Let $G = (V, E)$ be a finite, connected and simple graph. We denote the degree of a vertex $v$ in $G$ by $d_v$. 
A topological index of $G$ is a real number related to $G$. It does not depend on the labeling or pictorial representation of a graph. 
In the mathematical and chemical literature, several dozens of vertex-degree-based (VDB) graph
invariants such as  the first and second Zagreb index, the Albertson index and the Sombor index have been introduced and extensively studied \cite{1,2,4}.
Their general formula is 
\[
TI(G)=\sum_{ij\in E} F(d_i,d_j), 
\]
where $F(x,y)$ is some function with the property $F(x,y)=F(y,x)$.

The Sombor index, is defined as  
$$SO(G) =\sum_{uv\in E(G)}\sqrt{d_u^2+d_v^2}.$$
  
  This index soon attracted much attention, and its numerous mathematical \cite{9,10,11,13,14,12} and chemical
  \cite{17,18,19,Energy,Sombor,15,16,14} applications have been established. Recently, Gutman \cite{Gutman} intended to point out the great variety of Sombor-index-like VDB graph invariants that can
  be constructed by means of geometric arguments and so generated  a number of new
  Sombor-index-like VDB invariants, denoted below by $SO_1, SO_2,..., SO_6$.
  
$$SO_1(G) =\frac{1}{2}\sum_{uv\in E(G)}|d_u^2-d_v^2|.$$

$$SO_2(G) =\sum_{uv\in E(G)}\Big| \frac{d_u^2-d_v^2}{d_u^2+d_v^2} \Big|.$$

$$SO_3(G) =\sum_{uv\in E(G)}\sqrt{2}\frac{d_u^2+d_v^2}{d_u+d_v}\pi.$$

$$SO_4(G) =\frac{1}{2}\sum_{uv\in E(G)}\left(\frac{d_u^2+d_v^2}{d_u+d_v}\right)^2\pi.$$

$$SO_5(G) =\sum_{uv\in E(G)}\frac{2|d_u^2-d_v^2|}{\sqrt{2}+2\sqrt{d_u^2+d_v^2}}\pi.$$

$$SO_6(G) =\sum_{uv\in E(G)}\left[\frac{d_u^2-d_v^2}{\sqrt{2}+2\sqrt{d_u^2+d_v^2}}\right]^2\pi.$$

Polymer graphs  can be decomposed into subgraphs that we call monomer units. Blocks of graphs are particular examples of monomer units, but a monomer unit may consist of several blocks.

Let $G$ be a connected graph constructed from pairwise disjoint connected graphs
$G_1,\ldots ,G_k$ as follows. Select a vertex of $G_1$, a vertex of $G_2$, and identify these two vertices. Then continue in this manner inductively.  Note that the graph $G$ constructed in this way has a tree-like structure, the $G_i$'s being its building stones (see Figure \ref{Figure1}).  Usually we  say that $G$ is a polymer graph, obtained by point-attaching from $G_1,\ldots , G_k$ and that $G_i$'s are the monomer units of $G$. A particular case of this construction is the decomposition of a connected graph into blocks (see \cite{Deutsch,Nima0,Moster}).

\begin{figure}
	\begin{center}
		\psscalebox{0.6 0.6}
		{
			\begin{pspicture}(0,-4.819607)(13.664668,2.90118)
			\pscircle[linecolor=black, linewidth=0.04, dimen=outer](5.0985146,1.0603933){1.6}
			\pscustom[linecolor=black, linewidth=0.04]
			{
				\newpath
				\moveto(11.898515,0.66039336)
			}
			\pscustom[linecolor=black, linewidth=0.04]
			{
				\newpath
				\moveto(11.898515,0.26039338)
			}
			\pscustom[linecolor=black, linewidth=0.04]
			{
				\newpath
				\moveto(12.698514,0.66039336)
			}
			\pscustom[linecolor=black, linewidth=0.04]
			{
				\newpath
				\moveto(10.298514,1.0603933)
			}
			\pscustom[linecolor=black, linewidth=0.04]
			{
				\newpath
				\moveto(11.098515,-0.9396066)
			}
			\pscustom[linecolor=black, linewidth=0.04]
			{
				\newpath
				\moveto(11.098515,-0.9396066)
			}
			\pscustom[linecolor=black, linewidth=0.04]
			{
				\newpath
				\moveto(11.898515,0.66039336)
			}
			\pscustom[linecolor=black, linewidth=0.04]
			{
				\newpath
				\moveto(11.898515,-0.9396066)
			}
			\pscustom[linecolor=black, linewidth=0.04]
			{
				\newpath
				\moveto(11.898515,-0.9396066)
			}
			\pscustom[linecolor=black, linewidth=0.04]
			{
				\newpath
				\moveto(12.698514,-0.9396066)
			}
			\pscustom[linecolor=black, linewidth=0.04]
			{
				\newpath
				\moveto(12.698514,0.26039338)
			}
			\pscustom[linecolor=black, linewidth=0.04]
			{
				\newpath
				\moveto(14.298514,0.66039336)
				\closepath}
			\psbezier[linecolor=black, linewidth=0.04](11.598515,1.0203934)(12.220886,1.467607)(12.593457,1.262929)(13.268515,1.0203933715820312)(13.943572,0.7778577)(12.308265,0.90039337)(12.224765,0.10039337)(12.141264,-0.69960666)(10.976142,0.5731798)(11.598515,1.0203934)
			\psbezier[linecolor=black, linewidth=0.04](4.8362556,-3.2521083)(4.063277,-2.2959895)(4.6714916,-1.9655427)(4.891483,-0.99004078729821)(5.111474,-0.014538889)(5.3979383,-0.84551746)(5.373531,-1.8452196)(5.349124,-2.8449216)(5.6092343,-4.208227)(4.8362556,-3.2521083)
			\psbezier[linecolor=black, linewidth=0.04](8.198514,-2.0396066)(6.8114076,-1.3924998)(6.844908,-0.93520766)(5.8785143,-1.6996066284179687)(4.9121203,-2.4640057)(5.6385145,-3.4996066)(6.3385143,-2.8396065)(7.0385146,-2.1796067)(9.585621,-2.6867135)(8.198514,-2.0396066)
			\pscircle[linecolor=black, linewidth=0.04, dimen=outer](7.5785146,-3.6396067){1.18}
			\psdots[linecolor=black, dotsize=0.2](11.418514,0.7403934)
			\psdots[linecolor=black, dotsize=0.2](9.618514,1.5003934)
			\psdots[linecolor=black, dotsize=0.2](6.6585145,0.7403934)
			\psdots[linecolor=black, dotsize=0.2](3.5185144,0.96039337)
			\psdots[linecolor=black, dotsize=0.2](5.1185145,-0.51960665)
			\psdots[linecolor=black, dotsize=0.2](5.3985143,-2.5796065)
			\psdots[linecolor=black, dotsize=0.2](7.458514,-2.4596066)
			\rput[bl](8.878514,0.42039338){$G_i$}
			\rput[bl](7.478514,-4.1196065){$G_j$}
			\psbezier[linecolor=black, linewidth=0.04](0.1985144,0.22039337)(0.93261385,0.89943534)(2.1385605,0.6900083)(3.0785143,0.9403933715820313)(4.0184684,1.1907784)(3.248657,0.442929)(2.2785144,0.20039338)(1.3083719,-0.042142253)(-0.53558505,-0.45864862)(0.1985144,0.22039337)
			\psbezier[linecolor=black, linewidth=0.04](2.885918,1.4892112)(1.7389486,2.4304078)(-0.48852357,3.5744174)(0.5524718,2.1502930326916756)(1.5934672,0.7261687)(1.5427756,1.2830372)(2.5062277,1.2429687)(3.46968,1.2029002)(4.0328875,0.5480146)(2.885918,1.4892112)
			\psellipse[linecolor=black, linewidth=0.04, dimen=outer](9.038514,0.7403934)(2.4,0.8)
			\psbezier[linecolor=black, linewidth=0.04](9.399693,1.883719)(9.770389,2.812473)(12.016343,2.7533927)(13.011008,2.856550531577144)(14.005673,2.9597082)(13.727474,2.4925284)(12.761896,2.2324166)(11.796317,1.9723049)(9.028996,0.9549648)(9.399693,1.883719)
			\end{pspicture}
		}
	\end{center}
	\caption{\label{Figure1} A polymer graph with monomer units  $G_1,\ldots , G_k$.}
\end{figure}
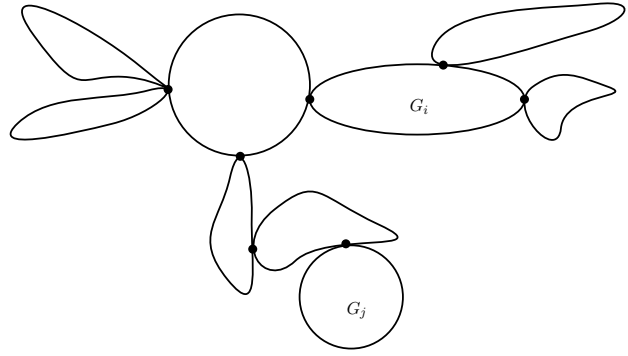

\medskip

In Section 2, We compute the Sombor-like degree-based  indices for certain graphs such as path, cycle, complete bipartite, friendship, ladder, book, Dutch-windmill and cactus chains. In Section 3, we consider polymers and study  some of  $SO_i(G)$ based their monomer units $G_i$. In Section 4,  we obtain some relations between indices $SO_i(G)$ ($i=2,3,4,5,6$) and index  $SO_1(G)$.  

\section{Sombor-like degree-based indices of certain graphs}
In this section, we compute the Sombor-index-like invariants  indices for certain graphs, such as paths, friendship graph, grid graphs and cactus chains. By the definitions of $SO_1(G)$, $SO_2(G)$, $SO_5(G)$ and $SO_6(G)$, these indices are $0$ for every regular graphs.
	
\subsection{Sombor-like degree-based indices of specific graphs}
 We begin with the following theorem:  
	\begin{theorem}\label{specific} 
		\begin{enumerate} 
\item[(i)] 
For every $n\in \mathbb{N}-\{1,2\}$, $SO_1(P_n)=3$.
\item[(ii)] 
For  every $m,n\in \mathbb{N}$, $SO_1(K_{1,n})=\frac{n(n-1)(n+1)}{2} $ and $SO_1(K_{m,n})=\frac{mn(m-n)(m+n)}{2},$ if $m\geq n$. 
\item[(iii)] 
For the wheel graph $W_{n+1}=C_{n}\vee K_1$, $SO_1(W_{n+1})=\frac{n(n-3)(n+3)}{2}.$
\item[(iv)] 
For the ladder graph $L_n=P_n\Box K_2$, ($n\geq 3$), $SO_1(L_n)=10.$
	\item[(v)] 
	For the friendship graph $F_n=K_1\vee nK_2$, $SO_1(F_n)=4n(n-1)(n+1).$

	\item[(vi)]
		For the book graph $B_n=K_{1,n}\square K_2$, and for $n\geq 3$,
		$SO_1(B_n)=n(n+3)(n-1).$
	
		\item[(vii)] For the Dutch windmill graph $D_n^{(m)}$, and for every $n\geq 3$ and $m\geq 2$,
		$$SO_1(D_n^{(m)})=4m(m-1)(m+1).$$
\end{enumerate} 
	\end{theorem}
		
	\begin{proof}
	The proof of all parts are easy and similar. For instance we state the proof of Part (i). 
	\begin{enumerate} 
	\item[(i)]
			There are two edges with endpoints of degree $1$ and $2$. Also there are $n-3$ edges with endpoints of degree $2$. Therefore we have the result. \qed
\end{enumerate}
	\end{proof}

	Similarly, we can compute $SO_2,\ldots,SO_6$ for these graphs. Table 1 and 2 show the exact values for these Sombor-like degree-based indices.

\[
\begin{footnotesize}
\small{
	\begin{tabular}{|c|c|c|c|}
	\hline
	$G$&$SO_2(G)$&$SO_3(G)$&$SO_4(G)$\\[0.3ex]
	\hline
	$P_n$&$\frac{6}{5}$&$\frac{10\sqrt{2}\pi}{3}+(n-3)\sqrt{2}(2\pi)$&$\frac{25\pi}{9}+(n-3)(2\pi)$\\
	\hline
	$K_{1,n}$&$\frac{n(n^2-1)}{n^2+1}$&$\frac{\sqrt{2}n\pi(n^2+1)}{n+1}$&$\frac{n\pi}{2}\left( \frac{n^2+1}{n+1}\right)^2$\\
	\hline
	$K_{m,n}$&$\frac{mn(m^2-n^2)}{m^2+n^2}$&$\frac{\sqrt{2}nm\pi(m^2+n^2)}{m+n}$&$\frac{nm\pi}{2}\left( \frac{n^2+m^2}{n+m}\right)^2$\\
	\hline
	$W_n$&$\frac{(n-1)(n^2-2n-8)}{n^2-2n+10}$&$\sqrt{2}\pi\left(\frac{(n-1)(n^2-2n+10)}{n+2}+3n-3\right)$&$\frac{\pi}{2}\left((n-1)\left(\frac{n^2-2n+10}{n+2}\right)^2+9n-9\right)$\\
	\hline
	$L_n$&$\frac{20}{13}$&$\frac{\sqrt{2}\pi}{5}(45n-48)$&$\frac{\pi}{50}(675n-924)$\\
	\hline
	$F_n$&$\frac{2n^3-2n}{n^2+1}$&$\sqrt{2}\pi\left(\frac{4n^3+2n^2+6n}{n+1} \right)$&$2n\pi\left(\frac{2n^4+5n^2+2n+3}{n^2+2n+1}\right)$\\
	\hline
	$B_n$&$\frac{2n^3+4n^2-6n}{n^2+2n+5}$&$\sqrt{2}\pi\left(3n+1+2n\left(\frac{n^2+2n+5}{n+3}\right)\right)$&$\frac{\pi}{2}\left(4n+(n+1)^2+2n\left(\frac{n^2+2n+5}{n+3}\right)^2\right)$\\
	\hline
	$D_n^{(m)}$&$\frac{2m(m^2-1)}{m^2+1}$&$\sqrt{2}\pi\left(2m(n-2)+4m\left(\frac{m^2+1}{m+1}\right)\right)$&$\frac{\pi}{2}\left(4m(n-2)+8m\left(\frac{m^2+1}{m+1}\right)^2\right)$\\
	\hline
	\end{tabular}}
\end{footnotesize}
\]
\begin{center}
	\noindent{Table 1.} The exact values of $SO_i(G)$ $(2\leq i\leq 4)$, for $m\geq n\geq 3$.
\end{center}

\[
\begin{footnotesize}
\small{
	\begin{tabular}{|c|c|c|}
	\hline
	$G$&$SO_5(G)$&$SO_6(G)$\\[0.3ex]
	\hline
	$P_n$&$\frac{12\pi}{\sqrt{2}+2\sqrt{5}}$&$\frac{18\pi}{(\sqrt{2}+2\sqrt{5})^2}$\\
	\hline
	$K_{1,n}$&$\frac{2n\pi(n^2-1)}{\sqrt{2}+2\sqrt{n^2+1}}$&$n\pi\left(\frac{n^2-1}{\sqrt{2}+2\sqrt{n^2+1}}\right)^2$\\
	\hline
	$K_{m,n}$&$\frac{2nm\pi(m^2-n^2)}{\sqrt{2}+2\sqrt{n^2+m^2}}$&$nm\pi\left(\frac{m^2-n^2}{\sqrt{2}+2\sqrt{n^2+m^2}}\right)^2$\\
	\hline
	$W_n$&$2(n-1)\pi\left(\frac{n^2-2n-8}{\sqrt{2}+2\sqrt{n^2-2n+10}}\right)$&$\pi(n-1)\left(\frac{n^2-2n-8}{\sqrt{2}+2\sqrt{n^2-2n+10}}\right)^2$\\
	\hline
	$L_n$&$\frac{40\pi}{\sqrt{2}+2\sqrt{13}}$&$\frac{100\pi}{(\sqrt{2}+2\sqrt{13})^2}$\\
	\hline
	$F_n$&$4n\pi\left(\frac{4n^2-4}{\sqrt{2}+2\sqrt{4n^2+4}}\right)$&$2n\pi\left(\frac{4n^2-4}{\sqrt{2}+2\sqrt{4n^2+4}}\right)^2$\\
	\hline
	$B_n$&$4n\pi\left(\frac{n^2+2n-3}{\sqrt{2}+2\sqrt{n^2+2n+5}}\right)$&$2n\pi\left(\frac{n^2+2n-3}{\sqrt{2}+2\sqrt{n^2+2n+5}}\right)^2$\\
	\hline
	$D_n^{(m)}$&$16m\pi\left(\frac{m^2-1}{\sqrt{2}+2\sqrt{4m^2+4}}\right)$&$32m\pi\left(\frac{m^2-1}{\sqrt{2}+2\sqrt{4m^2+4}}\right)^2$\\
	\hline
	\end{tabular}}
\end{footnotesize}
\]
\begin{center}
	\noindent{Table 2.} The exact values of $SO_5(G)$ and $SO_6(G)$, for $m\geq n\geq 3$.
\end{center}

Here, we consider the grid graph ($P_n\Box P_m$), and compute the Sombor-like degree-based indices for it.

\begin{theorem}
	Let $P_m\Box  P_n$ be the gird graph (Figure \ref{Gird}). For every $n\geq 6$ and $m \geq 6$,
	$$SO_1(P_m\Box P_n)=14m+14n-16,$$
	$$SO_2(P_m\Box P_n)=\frac{40}{13}+\frac{14m+14n-56}{25},$$
	$$SO_3(P_m\Box P_n)=\sqrt{2}\pi\left(\frac{104}{5}+\frac{50n+50m-200}{7}\right),$$
	$$SO_4(P_m\Box P_n)=\frac{\pi}{2}\left(\frac{1352}{25}+\frac{1250m+1250n-5000}{49}\right),$$
	$$SO_5(P_m\Box P_n)=2\pi\left(\frac{40}{\sqrt{2}+2\sqrt{13}}+\frac{14m+14n-56}{\sqrt{2}+10}\right),$$
	$$SO_6(P_m\Box P_n)=\pi\left(\frac{200}{\left(\sqrt{2}+2\sqrt{13}\right)^2}+\frac{98m+98n-392}{\left(\sqrt{2}+10\right)^2}\right).$$
\end{theorem}

\begin{proof}
	There are eight edges with endpoints of degree  $2$ and $3$ and there are $m+n-4$ edges with endpoints of degree $3$. Also there are $2m+2n-8$ edges with endpoints of degree $3$ and $4$ and there are $2nm-5n-5m-12$ edges with endpoints of degree $4$. Therefore 
 we have the result.			
	\qed
\end{proof}

\begin{figure}
	\begin{center}
		\psscalebox{0.5 0.5}
		{
			\begin{pspicture}(0,-7.6)(12.86,3.66)
			\psdots[linecolor=black, dotsize=0.4](1.63,2.03)
			\psdots[linecolor=black, dotsize=0.4](3.23,2.03)
			\psdots[linecolor=black, dotsize=0.4](4.83,2.03)
			\psdots[linecolor=black, dotsize=0.4](6.43,2.03)
			\psdots[linecolor=black, dotsize=0.4](8.03,2.03)
			\psline[linecolor=black, linewidth=0.08](1.63,2.03)(8.43,2.03)(8.43,2.03)
			\psdots[linecolor=black, dotsize=0.1](8.83,2.03)
			\psdots[linecolor=black, dotsize=0.1](9.23,2.03)
			\psdots[linecolor=black, dotsize=0.1](9.63,2.03)
			\psline[linecolor=black, linewidth=0.08](10.03,2.03)(11.63,2.03)(11.63,2.03)
			\psdots[linecolor=black, dotsize=0.4](10.43,2.03)
			\psdots[linecolor=black, dotsize=0.4](12.03,2.03)
			\psline[linecolor=black, linewidth=0.08](10.43,2.03)(12.03,2.03)(12.03,2.03)
			\psdots[linecolor=black, dotsize=0.4](1.63,0.43)
			\psdots[linecolor=black, dotsize=0.4](3.23,0.43)
			\psdots[linecolor=black, dotsize=0.4](4.83,0.43)
			\psdots[linecolor=black, dotsize=0.4](6.43,0.43)
			\psdots[linecolor=black, dotsize=0.4](8.03,0.43)
			\psdots[linecolor=black, dotsize=0.4](10.43,0.43)
			\psdots[linecolor=black, dotsize=0.4](12.03,0.43)
			\psdots[linecolor=black, dotsize=0.4](1.63,-1.17)
			\psdots[linecolor=black, dotsize=0.4](3.23,-1.17)
			\psdots[linecolor=black, dotsize=0.4](4.83,-1.17)
			\psdots[linecolor=black, dotsize=0.4](6.43,-1.17)
			\psdots[linecolor=black, dotsize=0.4](8.03,-1.17)
			\psdots[linecolor=black, dotsize=0.4](10.43,-1.17)
			\psdots[linecolor=black, dotsize=0.4](12.03,-1.17)
			\psdots[linecolor=black, dotsize=0.4](1.63,-2.77)
			\psdots[linecolor=black, dotsize=0.4](3.23,-2.77)
			\psdots[linecolor=black, dotsize=0.4](4.83,-2.77)
			\psdots[linecolor=black, dotsize=0.4](6.43,-2.77)
			\psdots[linecolor=black, dotsize=0.4](8.03,-2.77)
			\psdots[linecolor=black, dotsize=0.4](10.43,-2.77)
			\psdots[linecolor=black, dotsize=0.4](12.03,-2.77)
			\psline[linecolor=black, linewidth=0.08](1.63,-3.17)(1.63,2.03)(1.63,2.03)
			\psline[linecolor=black, linewidth=0.08](3.23,-3.17)(3.23,2.03)(3.23,2.03)
			\psline[linecolor=black, linewidth=0.08](4.83,-3.17)(4.83,2.03)(4.83,2.03)
			\psline[linecolor=black, linewidth=0.08](6.43,-3.17)(6.43,2.03)(6.43,2.03)
			\psline[linecolor=black, linewidth=0.08](8.03,-3.17)(8.03,2.03)(8.03,2.03)
			\psline[linecolor=black, linewidth=0.08](10.43,-3.17)(10.43,2.03)(10.43,2.03)
			\psline[linecolor=black, linewidth=0.08](12.03,-3.17)(12.03,2.03)(12.03,2.03)
			\psline[linecolor=black, linewidth=0.08](8.43,0.43)(1.63,0.43)(1.63,0.43)
			\psline[linecolor=black, linewidth=0.08](10.03,0.43)(12.03,0.43)(12.03,0.43)
			\psline[linecolor=black, linewidth=0.08](10.03,-1.17)(12.03,-1.17)(12.03,-1.17)
			\psline[linecolor=black, linewidth=0.08](10.03,-2.77)(12.03,-2.77)(12.03,-2.77)
			\psline[linecolor=black, linewidth=0.08](8.43,-1.17)(1.63,-1.17)(1.63,-1.17)
			\psline[linecolor=black, linewidth=0.08](8.43,-2.77)(1.63,-2.77)(1.63,-2.77)
			\psline[linecolor=black, linewidth=0.08](12.03,-4.77)(12.03,-6.77)(10.03,-6.77)(10.03,-6.77)
			\psline[linecolor=black, linewidth=0.08](10.43,-4.77)(10.43,-6.77)(10.43,-6.77)
			\psline[linecolor=black, linewidth=0.08](10.03,-5.17)(12.03,-5.17)(12.03,-5.17)
			\psline[linecolor=black, linewidth=0.08](8.43,-5.17)(1.63,-5.17)(1.63,-5.17)
			\psline[linecolor=black, linewidth=0.08](1.63,-4.77)(1.63,-6.77)(8.43,-6.77)(8.43,-6.77)
			\psline[linecolor=black, linewidth=0.08](8.03,-4.77)(8.03,-6.77)(8.03,-6.77)
			\psline[linecolor=black, linewidth=0.08](6.43,-4.77)(6.43,-6.77)(6.43,-6.77)
			\psline[linecolor=black, linewidth=0.08](4.83,-4.77)(4.83,-6.77)(4.83,-6.77)
			\psline[linecolor=black, linewidth=0.08](3.23,-4.77)(3.23,-6.77)(3.23,-6.77)
			\psdots[linecolor=black, dotsize=0.4](12.03,-5.17)
			\psdots[linecolor=black, dotsize=0.4](10.43,-5.17)
			\psdots[linecolor=black, dotsize=0.4](10.43,-6.77)
			\psdots[linecolor=black, dotsize=0.4](12.03,-6.77)
			\psdots[linecolor=black, dotsize=0.4](8.03,-5.17)
			\psdots[linecolor=black, dotsize=0.4](6.43,-5.17)
			\psdots[linecolor=black, dotsize=0.4](4.83,-5.17)
			\psdots[linecolor=black, dotsize=0.4](3.23,-5.17)
			\psdots[linecolor=black, dotsize=0.4](1.63,-5.17)
			\psdots[linecolor=black, dotsize=0.4](1.63,-6.77)
			\psdots[linecolor=black, dotsize=0.4](3.23,-6.77)
			\psdots[linecolor=black, dotsize=0.4](4.83,-6.77)
			\psdots[linecolor=black, dotsize=0.4](6.43,-6.77)
			\psdots[linecolor=black, dotsize=0.4](8.03,-6.77)
			\psdots[linecolor=black, dotsize=0.1](8.83,0.43)
			\psdots[linecolor=black, dotsize=0.1](9.23,0.43)
			\psdots[linecolor=black, dotsize=0.1](9.63,0.43)
			\psdots[linecolor=black, dotsize=0.1](8.83,-1.17)
			\psdots[linecolor=black, dotsize=0.1](9.23,-1.17)
			\psdots[linecolor=black, dotsize=0.1](9.63,-1.17)
			\psdots[linecolor=black, dotsize=0.1](8.83,-2.77)
			\psdots[linecolor=black, dotsize=0.1](9.23,-2.77)
			\psdots[linecolor=black, dotsize=0.1](9.63,-2.77)
			\psdots[linecolor=black, dotsize=0.1](10.43,-3.57)
			\psdots[linecolor=black, dotsize=0.1](10.43,-3.97)
			\psdots[linecolor=black, dotsize=0.1](10.43,-4.37)
			\psdots[linecolor=black, dotsize=0.1](12.03,-3.57)
			\psdots[linecolor=black, dotsize=0.1](12.03,-3.97)
			\psdots[linecolor=black, dotsize=0.1](12.03,-4.37)
			\psdots[linecolor=black, dotsize=0.1](8.83,-5.17)
			\psdots[linecolor=black, dotsize=0.1](9.23,-5.17)
			\psdots[linecolor=black, dotsize=0.1](9.63,-5.17)
			\psdots[linecolor=black, dotsize=0.1](8.83,-6.77)
			\psdots[linecolor=black, dotsize=0.1](9.23,-6.77)
			\psdots[linecolor=black, dotsize=0.1](9.63,-6.77)
			\psdots[linecolor=black, dotsize=0.1](8.03,-3.57)
			\psdots[linecolor=black, dotsize=0.1](8.03,-3.97)
			\psdots[linecolor=black, dotsize=0.1](8.03,-4.37)
			\psdots[linecolor=black, dotsize=0.1](6.43,-3.57)
			\psdots[linecolor=black, dotsize=0.1](6.43,-3.97)
			\psdots[linecolor=black, dotsize=0.1](6.43,-4.37)
			\psdots[linecolor=black, dotsize=0.1](4.83,-3.57)
			\psdots[linecolor=black, dotsize=0.1](4.83,-3.97)
			\psdots[linecolor=black, dotsize=0.1](4.83,-4.37)
			\psdots[linecolor=black, dotsize=0.1](3.23,-3.57)
			\psdots[linecolor=black, dotsize=0.1](3.23,-3.97)
			\psdots[linecolor=black, dotsize=0.1](3.23,-4.37)
			\psdots[linecolor=black, dotsize=0.1](1.63,-3.57)
			\psdots[linecolor=black, dotsize=0.1](1.63,-3.97)
			\psdots[linecolor=black, dotsize=0.1](1.63,-4.37)
			\psdots[linecolor=black, dotsize=0.1](8.83,-3.57)
			\psdots[linecolor=black, dotsize=0.1](9.23,-3.97)
			\psdots[linecolor=black, dotsize=0.1](9.63,-4.37)
			\rput[bl](7.096667,3.0166667){$P_n$}
			\rput[bl](0.23,-2.33){$P_m$}
			\psline[linecolor=blue, linewidth=0.04, linestyle=dotted, dotsep=0.10583334cm](0.83,3.63)(2.43,3.63)(2.43,-7.57)(0.83,-7.57)(0.83,3.63)(0.83,3.63)
			\psline[linecolor=red, linewidth=0.04, linestyle=dotted, dotsep=0.10583334cm](0.03,2.83)(12.83,2.83)(12.83,1.23)(0.03,1.23)(0.03,2.83)(0.03,2.83)
			\end{pspicture}
		}
	\end{center}
	\caption{Grid graph $P_m \Box P_n$ } \label{Gird}
\end{figure}
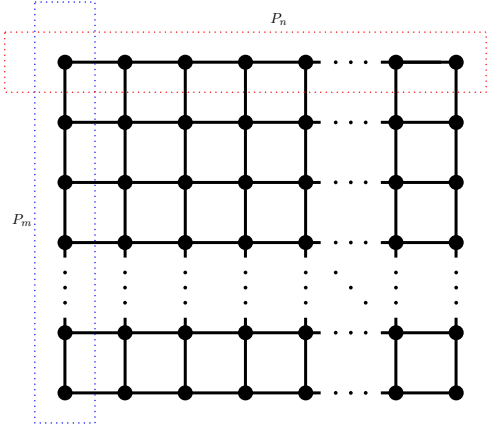

\subsection{Sombor-like degree-based indices of cactus chains} 	
	
In this subsection,  we consider a  class of simple linear polymers called cactus chains. Cactus graphs were first known as Husimi tree, they appeared in the scientific literature some sixty years ago in papers by Husimi and
Riddell concerned with cluster integrals in the theory of condensation in statistical mechanics \cite{9,12,14}. 
We refer the reader to papers \cite{Sombor,13} for some aspects of parameters of  cactus graphs.

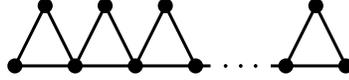
\begin{figure}
	\begin{center}
		\psscalebox{0.5 0.5}
		{
			\begin{pspicture}(0,-7.2)(9.194231,-5.205769)
			\psdots[linecolor=black, dotsize=0.1](5.7971153,-7.0028844)
			\psdots[linecolor=black, dotsize=0.1](6.1971154,-7.0028844)
			\psdots[linecolor=black, dotsize=0.1](6.5971155,-7.0028844)
			\psdots[linecolor=black, dotsize=0.4](7.3971157,-7.0028844)
			\psdots[linecolor=black, dotsize=0.4](8.197116,-5.4028845)
			\psdots[linecolor=black, dotsize=0.4](8.997115,-7.0028844)
			\psdots[linecolor=black, dotsize=0.4](4.9971156,-7.0028844)
			\psdots[linecolor=black, dotsize=0.4](4.1971154,-5.4028845)
			\psdots[linecolor=black, dotsize=0.4](3.3971155,-7.0028844)
			\psdots[linecolor=black, dotsize=0.4](2.5971155,-5.4028845)
			\psdots[linecolor=black, dotsize=0.4](1.7971154,-7.0028844)
			\psdots[linecolor=black, dotsize=0.4](0.9971155,-5.4028845)
			\psdots[linecolor=black, dotsize=0.4](0.19711548,-7.0028844)
			\psline[linecolor=black, linewidth=0.08](0.19711548,-7.0028844)(4.9971156,-7.0028844)(4.1971154,-5.4028845)(3.3971155,-7.0028844)(2.5971155,-5.4028845)(1.7971154,-7.0028844)(0.9971155,-5.4028845)(0.19711548,-7.0028844)(0.19711548,-7.0028844)
			\psline[linecolor=black, linewidth=0.08](7.3971157,-7.0028844)(8.197116,-5.4028845)(8.997115,-7.0028844)(7.3971157,-7.0028844)(7.3971157,-7.0028844)
			\psline[linecolor=black, linewidth=0.08](4.9971156,-7.0028844)(5.3971157,-7.0028844)(5.3971157,-7.0028844)
			\psline[linecolor=black, linewidth=0.08](7.3971157,-7.0028844)(6.9971156,-7.0028844)(6.9971156,-7.0028844)
			\end{pspicture}
		}
	\end{center}
	\caption{Chain triangular cactus $T_n$} \label{Chaintri}
\end{figure}

\begin{figure}
	\begin{center}
		\psscalebox{0.5 0.5}
		{
			\begin{pspicture}(0,-8.0)(9.194231,-4.405769)
			\psdots[linecolor=black, dotsize=0.1](5.7971153,-6.2028847)
			\psdots[linecolor=black, dotsize=0.1](6.1971154,-6.2028847)
			\psdots[linecolor=black, dotsize=0.1](6.5971155,-6.2028847)
			\psdots[linecolor=black, dotsize=0.4](7.3971157,-6.2028847)
			\psdots[linecolor=black, dotsize=0.4](8.197116,-4.6028843)
			\psdots[linecolor=black, dotsize=0.4](8.997115,-6.2028847)
			\psdots[linecolor=black, dotsize=0.4](4.9971156,-6.2028847)
			\psdots[linecolor=black, dotsize=0.4](4.1971154,-4.6028843)
			\psdots[linecolor=black, dotsize=0.4](3.3971155,-6.2028847)
			\psdots[linecolor=black, dotsize=0.4](2.5971155,-4.6028843)
			\psdots[linecolor=black, dotsize=0.4](1.7971154,-6.2028847)
			\psdots[linecolor=black, dotsize=0.4](0.9971155,-4.6028843)
			\psdots[linecolor=black, dotsize=0.4](0.19711548,-6.2028847)
			\psdots[linecolor=black, dotsize=0.4](0.9971155,-7.8028846)
			\psdots[linecolor=black, dotsize=0.4](2.5971155,-7.8028846)
			\psdots[linecolor=black, dotsize=0.4](4.1971154,-7.8028846)
			\psdots[linecolor=black, dotsize=0.4](8.197116,-7.8028846)
			\psline[linecolor=black, linewidth=0.08](7.3971157,-6.2028847)(8.197116,-4.6028843)(8.997115,-6.2028847)(8.197116,-7.8028846)(7.3971157,-6.2028847)(7.3971157,-6.2028847)
			\psline[linecolor=black, linewidth=0.08](4.9971156,-6.2028847)(4.1971154,-4.6028843)(3.3971155,-6.2028847)(2.5971155,-4.6028843)(1.7971154,-6.2028847)(0.9971155,-4.6028843)(0.19711548,-6.2028847)(0.9971155,-7.8028846)(1.7971154,-6.2028847)(2.5971155,-7.8028846)(3.3971155,-6.2028847)(4.1971154,-7.8028846)(4.9971156,-6.2028847)(4.9971156,-6.2028847)
			\end{pspicture}
		}
	\end{center}
	\caption{Para-chain square cactus $Q_n$} \label{paraChainsqu}
\end{figure}
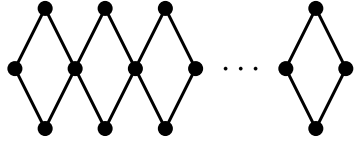

\begin{figure}
\begin{minipage}{7.5cm}
		\psscalebox{0.45 0.45}
		{
			\begin{pspicture}(0,-6.0)(12.394231,-2.405769)
			\psdots[linecolor=black, dotsize=0.4](0.19711548,-2.6028845)
			\psdots[linecolor=black, dotsize=0.4](1.7971154,-2.6028845)
			\psdots[linecolor=black, dotsize=0.4](0.19711548,-4.2028847)
			\psdots[linecolor=black, dotsize=0.4](1.7971154,-4.2028847)
			\psdots[linecolor=black, dotsize=0.4](3.3971155,-4.2028847)
			\psdots[linecolor=black, dotsize=0.4](4.9971156,-4.2028847)
			\psdots[linecolor=black, dotsize=0.4](6.5971155,-4.2028847)
			\psdots[linecolor=black, dotsize=0.4](1.7971154,-5.8028846)
			\psdots[linecolor=black, dotsize=0.4](3.3971155,-5.8028846)
			\psdots[linecolor=black, dotsize=0.4](3.3971155,-2.6028845)
			\psdots[linecolor=black, dotsize=0.4](4.9971156,-2.6028845)
			\psdots[linecolor=black, dotsize=0.4](6.5971155,-5.8028846)
			\psdots[linecolor=black, dotsize=0.4](4.9971156,-5.8028846)
			\psline[linecolor=black, linewidth=0.08](0.19711548,-4.2028847)(6.5971155,-4.2028847)(6.5971155,-4.2028847)
			\psline[linecolor=black, linewidth=0.08](0.19711548,-2.6028845)(1.7971154,-2.6028845)(1.7971154,-5.8028846)(3.3971155,-5.8028846)(3.3971155,-2.6028845)(4.9971156,-2.6028845)(4.9971156,-5.8028846)(6.5971155,-5.8028846)(6.5971155,-4.2028847)(6.5971155,-4.2028847)
			\psline[linecolor=black, linewidth=0.08](0.19711548,-4.2028847)(0.19711548,-2.6028845)(0.19711548,-2.6028845)
			\psline[linecolor=black, linewidth=0.08](6.9971156,-4.2028847)(6.5971155,-4.2028847)(6.5971155,-4.2028847)
			\psline[linecolor=black, linewidth=0.08](8.5971155,-4.2028847)(8.997115,-4.2028847)(8.997115,-4.2028847)
			\psdots[linecolor=black, dotsize=0.4](8.997115,-4.2028847)
			\psdots[linecolor=black, dotsize=0.4](8.997115,-2.6028845)
			\psdots[linecolor=black, dotsize=0.4](10.5971155,-2.6028845)
			\psdots[linecolor=black, dotsize=0.4](10.5971155,-4.2028847)
			\psdots[linecolor=black, dotsize=0.4](10.5971155,-5.8028846)
			\psdots[linecolor=black, dotsize=0.4](12.197116,-5.8028846)
			\psdots[linecolor=black, dotsize=0.4](12.197116,-4.2028847)
			\psline[linecolor=black, linewidth=0.08](8.997115,-4.2028847)(12.197116,-4.2028847)(12.197116,-5.8028846)(10.5971155,-5.8028846)(10.5971155,-2.6028845)(8.997115,-2.6028845)(8.997115,-4.2028847)(8.997115,-4.2028847)
			\psdots[linecolor=black, dotsize=0.1](7.3971157,-4.2028847)
			\psdots[linecolor=black, dotsize=0.1](7.7971153,-4.2028847)
			\psdots[linecolor=black, dotsize=0.1](8.197116,-4.2028847)
			\end{pspicture}
		}
\end{minipage}
\begin{minipage}{7.5cm}
		\psscalebox{0.45 0.45}
		{
			\begin{pspicture}(0,-6.8)(13.194231,-1.605769)
			\psdots[linecolor=black, dotsize=0.4](2.1971154,-1.8028846)
			\psdots[linecolor=black, dotsize=0.4](2.1971154,-4.2028847)
			\psdots[linecolor=black, dotsize=0.4](2.5971155,-3.0028846)
			\psdots[linecolor=black, dotsize=0.4](3.3971155,-3.0028846)
			\psdots[linecolor=black, dotsize=0.4](3.7971156,-4.2028847)
			\psdots[linecolor=black, dotsize=0.4](3.7971156,-1.8028846)
			\psdots[linecolor=black, dotsize=0.4](5.3971157,-1.8028846)
			\psdots[linecolor=black, dotsize=0.4](5.7971153,-3.0028846)
			\psdots[linecolor=black, dotsize=0.4](5.3971157,-4.2028847)
			\psdots[linecolor=black, dotsize=0.4](0.59711546,-1.8028846)
			\psdots[linecolor=black, dotsize=0.4](0.19711548,-3.0028846)
			\psdots[linecolor=black, dotsize=0.4](0.59711546,-4.2028847)
			\psdots[linecolor=black, dotsize=0.4](1.7971154,-5.4028845)
			\psdots[linecolor=black, dotsize=0.4](2.1971154,-6.6028843)
			\psdots[linecolor=black, dotsize=0.4](3.7971156,-6.6028843)
			\psdots[linecolor=black, dotsize=0.4](4.1971154,-5.4028845)
			\psdots[linecolor=black, dotsize=0.4](6.9971156,-4.2028847)
			\psdots[linecolor=black, dotsize=0.4](4.9971156,-5.4028845)
			\psdots[linecolor=black, dotsize=0.4](5.3971157,-6.6028843)
			\psdots[linecolor=black, dotsize=0.4](7.3971157,-5.4028845)
			\psdots[linecolor=black, dotsize=0.4](6.9971156,-6.6028843)
			\psdots[linecolor=black, dotsize=0.4](10.997115,-1.8028846)
			\psdots[linecolor=black, dotsize=0.4](10.997115,-4.2028847)
			\psdots[linecolor=black, dotsize=0.4](11.397116,-3.0028846)
			\psdots[linecolor=black, dotsize=0.4](12.5971155,-4.2028847)
			\psdots[linecolor=black, dotsize=0.4](9.397116,-1.8028846)
			\psdots[linecolor=black, dotsize=0.4](8.997115,-3.0028846)
			\psdots[linecolor=black, dotsize=0.4](9.397116,-4.2028847)
			\psdots[linecolor=black, dotsize=0.4](10.5971155,-5.4028845)
			\psdots[linecolor=black, dotsize=0.4](10.997115,-6.6028843)
			\psdots[linecolor=black, dotsize=0.4](12.5971155,-6.6028843)
			\psdots[linecolor=black, dotsize=0.4](12.997115,-5.4028845)
			\psline[linecolor=black, linewidth=0.08](6.9971156,-4.2028847)(0.59711546,-4.2028847)(0.19711548,-3.0028846)(0.59711546,-1.8028846)(2.1971154,-1.8028846)(2.5971155,-3.0028846)(2.1971154,-4.2028847)(1.7971154,-5.4028845)(2.1971154,-6.6028843)(3.7971156,-6.6028843)(4.1971154,-5.4028845)(3.7971156,-4.2028847)(3.3971155,-3.0028846)(3.7971156,-1.8028846)(5.3971157,-1.8028846)(5.7971153,-3.0028846)(5.3971157,-4.2028847)(4.9971156,-5.4028845)(5.3971157,-6.6028843)(6.9971156,-6.6028843)(7.3971157,-5.4028845)(6.9971156,-4.2028847)(6.9971156,-4.2028847)
			\psline[linecolor=black, linewidth=0.08](9.397116,-4.2028847)(12.5971155,-4.2028847)(12.997115,-5.4028845)(12.5971155,-6.6028843)(10.997115,-6.6028843)(10.5971155,-5.4028845)(11.397116,-3.0028846)(10.997115,-1.8028846)(9.397116,-1.8028846)(8.997115,-3.0028846)(9.397116,-4.2028847)(9.397116,-4.2028847)
			\psline[linecolor=black, linewidth=0.08](7.3971157,-4.2028847)(6.9971156,-4.2028847)(6.9971156,-4.2028847)
			\psline[linecolor=black, linewidth=0.08](8.997115,-4.2028847)(9.397116,-4.2028847)(9.397116,-4.2028847)
			\psdots[linecolor=black, dotsize=0.1](8.197116,-4.2028847)
			\psdots[linecolor=black, dotsize=0.1](7.7971153,-4.2028847)
			\psdots[linecolor=black, dotsize=0.1](8.5971155,-4.2028847)
			\end{pspicture}
		}
	\end{minipage}
	\caption{Para-chain square cactus $O_n$ and Ortho-chain graph $O_n^h$ } \label{ortho-ohn}
\end{figure}
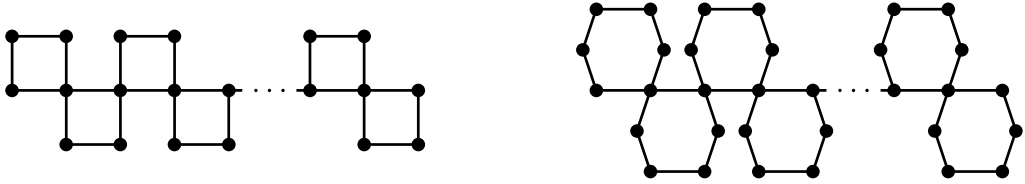

\begin{figure}
\begin{minipage}{7.5cm}
		\psscalebox{0.45 0.45}
		{
			\begin{pspicture}(0,-5.6)(16.794231,-2.805769)
			\psdots[linecolor=black, dotsize=0.4](2.1971154,-3.0028846)
			\psdots[linecolor=black, dotsize=0.4](2.1971154,-5.4028845)
			\psdots[linecolor=black, dotsize=0.4](2.9971154,-4.2028847)
			\psdots[linecolor=black, dotsize=0.4](2.9971154,-4.2028847)
			\psdots[linecolor=black, dotsize=0.4](3.7971153,-5.4028845)
			\psdots[linecolor=black, dotsize=0.4](3.7971153,-3.0028846)
			\psdots[linecolor=black, dotsize=0.4](4.9971156,-3.0028846)
			\psdots[linecolor=black, dotsize=0.4](5.7971153,-4.2028847)
			\psdots[linecolor=black, dotsize=0.4](4.9971156,-5.4028845)
			\psdots[linecolor=black, dotsize=0.4](0.9971154,-3.0028846)
			\psdots[linecolor=black, dotsize=0.4](0.19711538,-4.2028847)
			\psdots[linecolor=black, dotsize=0.4](0.9971154,-5.4028845)
			\psdots[linecolor=black, dotsize=0.4](7.7971153,-3.0028846)
			\psdots[linecolor=black, dotsize=0.4](7.7971153,-5.4028845)
			\psdots[linecolor=black, dotsize=0.4](8.5971155,-4.2028847)
			\psdots[linecolor=black, dotsize=0.4](8.5971155,-4.2028847)
			\psdots[linecolor=black, dotsize=0.4](6.5971155,-3.0028846)
			\psdots[linecolor=black, dotsize=0.4](5.7971153,-4.2028847)
			\psdots[linecolor=black, dotsize=0.4](6.5971155,-5.4028845)
			\psdots[linecolor=black, dotsize=0.4](12.997115,-3.0028846)
			\psdots[linecolor=black, dotsize=0.4](12.997115,-5.4028845)
			\psdots[linecolor=black, dotsize=0.4](13.797115,-4.2028847)
			\psdots[linecolor=black, dotsize=0.4](13.797115,-4.2028847)
			\psdots[linecolor=black, dotsize=0.4](14.5971155,-5.4028845)
			\psdots[linecolor=black, dotsize=0.4](14.5971155,-3.0028846)
			\psdots[linecolor=black, dotsize=0.4](15.797115,-3.0028846)
			\psdots[linecolor=black, dotsize=0.4](16.597115,-4.2028847)
			\psdots[linecolor=black, dotsize=0.4](15.797115,-5.4028845)
			\psdots[linecolor=black, dotsize=0.4](11.797115,-3.0028846)
			\psdots[linecolor=black, dotsize=0.4](10.997115,-4.2028847)
			\psdots[linecolor=black, dotsize=0.4](11.797115,-5.4028845)
			\psdots[linecolor=black, dotsize=0.4](8.5971155,-4.2028847)
			\psdots[linecolor=black, dotsize=0.4](8.5971155,-4.2028847)
			\psline[linecolor=black, linewidth=0.08](0.19711538,-4.2028847)(0.9971154,-3.0028846)(2.1971154,-3.0028846)(2.9971154,-4.2028847)(3.7971153,-3.0028846)(4.9971156,-3.0028846)(5.7971153,-4.2028847)(6.5971155,-3.0028846)(7.7971153,-3.0028846)(8.5971155,-4.2028847)(7.7971153,-5.4028845)(6.5971155,-5.4028845)(5.7971153,-4.2028847)(4.9971156,-5.4028845)(3.7971153,-5.4028845)(2.9971154,-4.2028847)(2.1971154,-5.4028845)(0.9971154,-5.4028845)(0.19711538,-4.2028847)(0.19711538,-4.2028847)
			\psline[linecolor=black, linewidth=0.08](10.997115,-4.2028847)(11.797115,-3.0028846)(12.997115,-3.0028846)(13.797115,-4.2028847)(14.5971155,-3.0028846)(15.797115,-3.0028846)(16.597115,-4.2028847)(15.797115,-5.4028845)(14.5971155,-5.4028845)(13.797115,-4.2028847)(12.997115,-5.4028845)(11.797115,-5.4028845)(10.997115,-4.2028847)(10.997115,-4.2028847)
			\psdots[linecolor=black, dotsize=0.1](9.397116,-4.2028847)
			\psdots[linecolor=black, dotsize=0.1](9.797115,-4.2028847)
			\psdots[linecolor=black, dotsize=0.1](10.197115,-4.2028847)
			\end{pspicture}
		}
	\end{minipage}
	\hspace{1.02cm}
	\begin{minipage}{7.5cm} 
		\psscalebox{0.45 0.40}
		{
			\begin{pspicture}(0,-6.4)(12.394231,-0.40576905)
			\psdots[linecolor=black, dotsize=0.4](0.9971155,-0.60288453)
			\psdots[linecolor=black, dotsize=0.4](1.7971154,-1.8028846)
			\psdots[linecolor=black, dotsize=0.4](1.7971154,-3.4028845)
			\psdots[linecolor=black, dotsize=0.4](0.9971155,-4.6028843)
			\psdots[linecolor=black, dotsize=0.4](0.19711548,-1.8028846)
			\psdots[linecolor=black, dotsize=0.4](0.19711548,-3.4028845)
			\psdots[linecolor=black, dotsize=0.4](2.5971155,-2.2028844)
			\psdots[linecolor=black, dotsize=0.4](3.3971155,-3.4028845)
			\psdots[linecolor=black, dotsize=0.4](1.7971154,-5.0028844)
			\psdots[linecolor=black, dotsize=0.4](3.3971155,-5.0028844)
			\psdots[linecolor=black, dotsize=0.4](2.5971155,-6.2028847)
			\psdots[linecolor=black, dotsize=0.4](4.1971154,-0.60288453)
			\psdots[linecolor=black, dotsize=0.4](4.9971156,-1.8028846)
			\psdots[linecolor=black, dotsize=0.4](4.9971156,-3.4028845)
			\psdots[linecolor=black, dotsize=0.4](4.1971154,-4.6028843)
			\psdots[linecolor=black, dotsize=0.4](3.3971155,-1.8028846)
			\psdots[linecolor=black, dotsize=0.4](3.3971155,-3.4028845)
			\psdots[linecolor=black, dotsize=0.4](5.7971153,-2.2028844)
			\psdots[linecolor=black, dotsize=0.4](6.5971155,-3.4028845)
			\psdots[linecolor=black, dotsize=0.4](4.9971156,-5.0028844)
			\psdots[linecolor=black, dotsize=0.4](6.5971155,-5.0028844)
			\psdots[linecolor=black, dotsize=0.4](5.7971153,-6.2028847)
			\psdots[linecolor=black, dotsize=0.1](7.3971157,-3.4028845)
			\psdots[linecolor=black, dotsize=0.1](7.7971153,-3.4028845)
			\psdots[linecolor=black, dotsize=0.1](8.197116,-3.4028845)
			\psdots[linecolor=black, dotsize=0.4](9.797115,-0.60288453)
			\psdots[linecolor=black, dotsize=0.4](10.5971155,-1.8028846)
			\psdots[linecolor=black, dotsize=0.4](10.5971155,-3.4028845)
			\psdots[linecolor=black, dotsize=0.4](9.797115,-4.6028843)
			\psdots[linecolor=black, dotsize=0.4](8.997115,-1.8028846)
			\psdots[linecolor=black, dotsize=0.4](8.997115,-3.4028845)
			\psdots[linecolor=black, dotsize=0.4](11.397116,-2.2028844)
			\psdots[linecolor=black, dotsize=0.4](12.197116,-3.4028845)
			\psdots[linecolor=black, dotsize=0.4](10.5971155,-5.0028844)
			\psdots[linecolor=black, dotsize=0.4](12.197116,-5.0028844)
			\psdots[linecolor=black, dotsize=0.4](11.397116,-6.2028847)
			\psline[linecolor=black, linewidth=0.08](0.9971155,-0.60288453)(1.7971154,-1.8028846)(1.7971154,-5.0028844)(2.5971155,-6.2028847)(3.3971155,-5.0028844)(3.3971155,-1.8028846)(4.1971154,-0.60288453)(4.9971156,-1.8028846)(4.9971156,-5.0028844)(5.7971153,-6.2028847)(6.5971155,-5.0028844)(6.5971155,-3.4028845)(5.7971153,-2.2028844)(4.9971156,-3.4028845)(4.1971154,-4.6028843)(2.5971155,-2.2028844)(0.9971155,-4.6028843)(0.19711548,-3.4028845)(0.19711548,-1.8028846)(0.9971155,-0.60288453)(0.9971155,-0.60288453)
			\psline[linecolor=black, linewidth=0.08](11.397116,-2.2028844)(9.797115,-4.6028843)(8.997115,-3.4028845)(8.997115,-1.8028846)(9.797115,-0.60288453)(10.5971155,-1.8028846)(10.5971155,-5.0028844)(11.397116,-6.2028847)(12.197116,-5.0028844)(12.197116,-3.4028845)(11.397116,-2.2028844)(11.397116,-2.2028844)
			\end{pspicture}
		}
	\end{minipage}
		\caption{Para-chain  $L_n$ and Meta-chain  $M_n$} \label{metaChainMn}
	\end{figure}
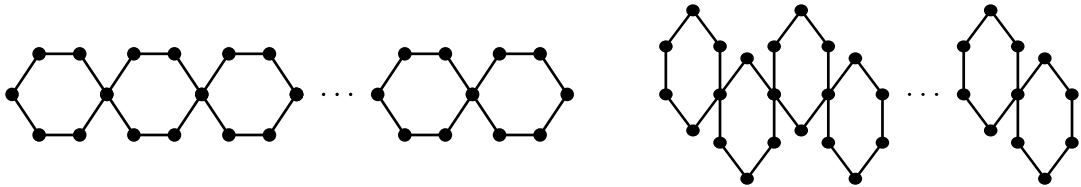

	\begin{theorem}
		\begin{enumerate} 
			\item[(i)] 
			Let $T_n$ be the chain triangular graph (See Figure \ref{Chaintri}) of order $n$. Then for every $n\geq 2$,
$SO_1(T_n)=12n.$

\item[(ii)] 
Let $Q_n$ be the para-chain square cactus graph (See Figure \ref{paraChainsqu}) of order $n$. Then for every $n\geq 2$,
$SO_1(Q_n)=24n-24.$

\item[(iii)] 
Let $O_n$ be the para-chain square cactus (See Figure \ref{ortho-ohn}) graph of order $n$. Then for every $n\geq 2$,
$SO_1(O_n)=12n.$

\item[(iv)] 
 Let $O_n^h$ be the Ortho-chain graph (See Figure \ref{ortho-ohn}) of order $n$. Then for every $n\geq 2$,
$SO_1(O_n^h)=12n.$

\item[(v)] 
Let $L_n$ be the para-chain hexagonal
cactus graph (See Figure \ref{metaChainMn}) of order $n$. Then for every $n\geq 2$,
$SO_1(L_n)=24n-24.$

\item[(vi)] 
Let $M_n$ be the Meta-chain hexagonal
cactus graph (See Figure \ref{metaChainMn}) of order $n$. Then for every $n\geq 2$,
$SO_1(M_n)=24n-24.$
\end{enumerate} 
	\end{theorem}

		\begin{proof}
			\begin{enumerate} 
				\item[(i)] 
	There are two edges with endpoints of degree $2$. Also there are $2n$ edges with endpoints of degree $2$ and $4$ and there are $n-2$ edges with endpoints of degree $4$. Therefore 
	 we have the result.	
	\item[(ii)]
	There are four edges with endpoints of degree $2$. Also there are $4n-4$ edges with endpoints of degree $2$ and $4$. Therefore 
	 the result follows.			
	\item[(iii)]	
	There are $n+2$ edges with endpoints of degree $2$. Also there are $2n$ edges with endpoints of degree $2$ and $4$ and there are $n-2$ edges with endpoints of degree $4$. Therefore 
	 we have the result.	
		\item[(iv)] 
	There are $3n+2$ edges with endpoints of degree $2$. Also there are $2n$ edges with endpoints of degree 2 and 4 and there are $n-2$ edges with endpoints of degree 4. Therefore 
 the result follows.
\item[(v)] 
	There are $2n+4$ edges with endpoints of degree $2$. Also there are $4n-4$ edges with endpoints of degree 2 and 4. Therefore 
	 we have the result.	
	\item[(vi)] 
	There are $2n+4$ edges with endpoints of degree $2$. Also there are $4n-4$ edges with endpoints of degree $2$ and $4$. Therefore 
	the result follows. \qed
	\end{enumerate} 
		\end{proof}

Similarly, we can compute $SO_2,\ldots,SO_6$ for cactus chains. Table 3 shows the exact value for these Sombor-like degree-based indices.

\[
\begin{footnotesize}
\small{
	\begin{tabular}{|c|c|c|c|c|c|}
	\hline
	$G$&$SO_2(G)$&$SO_3(G)$&$SO_4(G)$&$SO_5(G)$&$SO_6(G)$\\[0.3ex]
	\hline
	$T_n$&$\frac{6n}{5}$&$3\sqrt{2}\pi\left(\frac{44n-36}{3}\right)$&$\frac{\pi}{2}\left(\frac{776n-1080}{9}\right)$&$\frac{48n\pi}{\sqrt{2}+2\sqrt{20}}$&$\frac{288n\pi}{\left(\sqrt{2}+2\sqrt{20}\right)^2}$\\
	\hline
	$Q_n$&$\frac{12n-12}{5}$&$\sqrt{2}\pi\left(\frac{40n-16}{3}\right)$&$\frac{\pi}{2}\left(\frac{400n-256}{9}\right)$&$\frac{24\pi(4n-4)}{\sqrt{2}+2\sqrt{20}}$&$\frac{144\pi(4n-4)}{\left(\sqrt{2}+2\sqrt{20}\right)^2}$\\
	\hline
	$O_n$&$\frac{6n}{5}$&$\sqrt{2}\pi\left(\frac{38n-12}{3}\right)$&$\frac{\pi}{2}\left(\frac{380n-216}{9}\right)$&$\frac{48n\pi}{\sqrt{2}+2\sqrt{20}}$&$\frac{288n\pi}{\left(\sqrt{2}+2\sqrt{20}\right)^2}$\\
	\hline
	$O_n^h$&$\frac{6n}{5}$&$\sqrt{2}\pi\left(\frac{50n-12}{3}\right)$&$\frac{\pi}{2}\left(\frac{452n-216}{9}\right)$&$\frac{48n\pi}{\sqrt{2}+2\sqrt{20}}$&$\frac{288n\pi}{\left(\sqrt{2}+2\sqrt{20}\right)^2}$\\
	\hline
	$L_n$&$\frac{12n-12}{5}$&$\sqrt{2}\pi\left(\frac{52n-26}{3}\right)$&$\frac{\pi}{2}\left(\frac{472n-256}{9}\right)$&$\frac{24\pi(4n-4)}{\sqrt{2}+2\sqrt{20}}$&$\frac{144\pi(4n-4)}{\left(\sqrt{2}+2\sqrt{20}\right)^2}$\\
	\hline
	$M_n$&$\frac{12n-12}{5}$&$\sqrt{2}\pi\left(\frac{52n-26}{3}\right)$&$\frac{\pi}{2}\left(\frac{472n-256}{9}\right)$&$\frac{24\pi(4n-4)}{\sqrt{2}+2\sqrt{20}}$&$\frac{144\pi(4n-4)}{\left(\sqrt{2}+2\sqrt{20}\right)^2}$\\
	\hline
	\end{tabular}}
\end{footnotesize}
\]
\begin{center}
	\noindent{Table 3.} The exact values of $SO_i(G)$ $(2\leq i\leq 6)$, for $n\geq 2$.
\end{center}

\begin{corollary}
	Meta-chain hexagonal
	cactus graphs  and para-chain hexagonal
	cactus graphs of the same order, have the same Sombor-like degree-based indices.
\end{corollary}

\section{Study of $SO_1(G)$ for  polymers and their monomer's unit}

In this section, we study  the index $SO_1(G) =\frac{1}{2}\sum_{uv\in E(G)}|d_u^2-d_v^2|$, where $G$ is a polymer graph. 
We begin with the following proposition which gives an upper bound for $SO_1(G-e)$:

	\begin{proposition}\label{pro(G-e)}
		Let $G=(V,E)$ be a non-regular graph and $e=uv\in E$. Also let $d_w$, $\Delta$ and $\delta$, be the degree of vertex $w$, maximum degree and minimum degree of vertices  in $G$, respectively. Then,
		$$SO_1(G) > SO_1(G-e) + \frac{1}{2}(\delta^2-\Delta^2). $$
	\end{proposition}
	
	\begin{proof}
		First we remove edge $e$ and find $SO_1(G-e)$. Obviously, by adding edge $e$ to $G-e$ and $\frac{1}{2}|d_u^2-d_v^2|$ to $SO_1(G-e)$, the value $SO_1(G)$ is greater than that. So we have
	\begin{align*}
	SO_1(G) > SO_1(G-e) + \frac{1}{2}|d_u^2-d_v^2|
	\geq SO_1(G-e) + \frac{1}{2}(d_u^2-d_v^2),
	\end{align*}
		and therefore we have the result.
		\qed
	\end{proof}

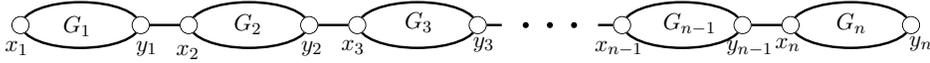
\begin{figure}
	\begin{center}
		\psscalebox{0.8 0.8}
		{
			\begin{pspicture}(0,-4.08)(15.436667,-3.12)
			\psellipse[linecolor=black, linewidth=0.04, dimen=outer](1.2533334,-3.52)(1.0,0.4)
			\psellipse[linecolor=black, linewidth=0.04, dimen=outer](4.0533333,-3.52)(1.0,0.4)
			\psellipse[linecolor=black, linewidth=0.04, dimen=outer](6.853334,-3.52)(1.0,0.4)
			\psellipse[linecolor=black, linewidth=0.04, dimen=outer](11.253334,-3.52)(1.0,0.4)
			\psellipse[linecolor=black, linewidth=0.04, dimen=outer](14.053333,-3.52)(1.0,0.4)
			\psline[linecolor=black, linewidth=0.04](2.2533333,-3.52)(3.0533335,-3.52)(3.0533335,-3.52)
			\psline[linecolor=black, linewidth=0.04](5.0533333,-3.52)(5.8533335,-3.52)(5.8533335,-3.52)
			\psline[linecolor=black, linewidth=0.04](12.253333,-3.52)(13.053333,-3.52)(13.053333,-3.52)
			\psdots[linecolor=black, dotstyle=o, dotsize=0.3, fillcolor=white](2.2533333,-3.52)
			\psdots[linecolor=black, dotstyle=o, dotsize=0.3, fillcolor=white](0.25333345,-3.52)
			\psdots[linecolor=black, dotstyle=o, dotsize=0.3, fillcolor=white](3.0533335,-3.52)
			\psdots[linecolor=black, dotstyle=o, dotsize=0.3, fillcolor=white](5.0533333,-3.52)
			\psdots[linecolor=black, dotstyle=o, dotsize=0.3, fillcolor=white](5.8533335,-3.52)
			\psdots[linecolor=black, dotstyle=o, dotsize=0.3, fillcolor=white](12.253333,-3.52)
			\psdots[linecolor=black, dotstyle=o, dotsize=0.3, fillcolor=white](13.053333,-3.52)
			\psdots[linecolor=black, dotstyle=o, dotsize=0.3, fillcolor=white](15.053333,-3.52)
			\rput[bl](0.0,-4.0133333){$x_1$}
			\rput[bl](2.8400002,-4.08){$x_2$}
			\rput[bl](5.5866666,-4.0){$x_3$}
			\rput[bl](2.1733334,-4.0266666){$y_1$}
			\rput[bl](4.92,-4.0266666){$y_2$}
			\rput[bl](7.7733335,-3.96){$y_3$}
			\rput[bl](0.9600001,-3.7066667){$G_1$}
			\rput[bl](3.8000002,-3.6666667){$G_2$}
			\rput[bl](6.64,-3.64){$G_3$}
			\psline[linecolor=black, linewidth=0.04](8.253333,-3.52)(7.8533335,-3.52)(7.8533335,-3.52)
			\psline[linecolor=black, linewidth=0.04](9.853333,-3.52)(10.253333,-3.52)(10.253333,-3.52)
			\psdots[linecolor=black, dotstyle=o, dotsize=0.3, fillcolor=white](7.8533335,-3.52)
			\psdots[linecolor=black, dotstyle=o, dotsize=0.3, fillcolor=white](10.253333,-3.52)
			\psdots[linecolor=black, dotsize=0.1](8.653334,-3.52)
			\psdots[linecolor=black, dotsize=0.1](9.053333,-3.52)
			\psdots[linecolor=black, dotsize=0.1](9.453334,-3.52)
			\rput[bl](12.8,-3.96){$x_n$}
			\rput[bl](15.026667,-3.9466667){$y_n$}
			\rput[bl](11.986667,-4.0266666){$y_{n-1}$}
			\rput[bl](10.933333,-3.68){$G_{n-1}$}
			\rput[bl](9.8,-4.04){$x_{n-1}$}
			\rput[bl](13.8133335,-3.6533334){$G_n$}
			\end{pspicture}
		}
	\end{center}
	\caption{Link of $n$ graphs $G_1,G_2, \ldots , G_n$} \label{link-n}
\end{figure}

Here, we study the Sombor-like degree-based indices of some polymer graphs.  

\begin{theorem} \label{thm-link}
 Let $G$ be a polymer graph with composed of 	monomers $\{G_i\}_{i=1}^k$ with respect to the vertices $\{x_i, y_i\}_{i=1}^k$, and $\Delta_i$ and $\delta_i$ be the maximum degree and minimum degree of vertices  in $G_i$, respectively.  Let $G=L(G_1,...,G_n)$ be the link of graphs  (see Figure \ref{link-n}).  If $G$ and $L(G_1,...,G_i)$ are  non-regular for $i\geq 2$, then,
	$$SO_1(G)>\sum_{i=1}^{k}SO_1(G_i)+\frac{1}{2}\sum_{i=1}^{k-1}\delta_i^2-\frac{1}{2}\sum_{i=2}^{k}\Delta_i^2.$$
\end{theorem}

\begin{proof}
	First we remove edge $y_1x_2$ (Figure \ref{link-n}). Similar to the proof of  Proposition \ref{pro(G-e)}, we have
	$$SO_1(G) > SO_1(G-y_1x_2) + \frac{1}{2}(d_{y_1}^2-d_{x_2}^2).$$
	Let $G^{\prime}$ be the link graph related to graphs $\{G_i\}_{i=2}^k$ with respect to the vertices $\{x_i, y_i\}_{i=2}^k$. Then we have,
	$$SO_1(G-y_1x_2)=SO_1(G_1)+SO_1(G^{\prime}),$$
	and therefore,
		\begin{align*}
SO_1(G) &> SO_1(G_1)+SO_1(G^{\prime}) + \frac{1}{2}(d_{y_1}^2-d_{x_2}^2)\\
&\geq SO_1(G_1)+SO_1(G^{\prime}) + \frac{1}{2}(\delta_1^2-\Delta_2^2) .
	\end{align*}
	By continuing this process, we have the result.
	\qed
\end{proof}

As an immediate result of theorem \ref{thm-link}, we have:

\begin{corollary}
Let $G$ be a polymer graph with composed of 	monomers $\{G_i\}_{i=1}^k$ with respect to the vertices $\{x_i, y_i\}_{i=1}^k$.   If $G$,  $L(G_1,...,G_i)$ are  non-regular for $i\geq 2$, and $\Delta$ and $\delta$ are the maximum degree and the  minimum degree of vertices  in $G$.  Then,
	$$SO_1(G)>\frac{k-1}{2}\left(\delta^2-\Delta^2\right)+\sum_{i=1}^{k}SO_1(G_i).$$
\end{corollary}

\section{Relations between $SO_i(G)$ $(i=2,3,4,5,6)$ and  $SO_1(G)$} 

In this section, we obtain bounds for $SO_i(G)$ $(i=2,3,4,5,6)$ based $SO_1(G)$.   
We begin with the following proposition which gives a lower and upper bounds for $SO_2(G)$ based $SO_1(G)$.

\begin{proposition}\label{pro-so1-so2}
	If  $G$ is  a graph with minimum degree $\delta$ and maximum degree $\Delta$, then
	$$\frac{SO_1(G)}{\Delta^2} \leq SO_2(G)\leq \frac{SO_1(G)}{\delta^2}.$$
\end{proposition}

\begin{proof}
	First we find the lower bound. We have
	\begin{align*}
	SO_2(G) &= \sum_{uv\in E(G)}\Big| \frac{d_u^2-d_v^2}{d_u^2+d_v^2} \Big|\geq \sum_{uv\in E(G)} \frac{|d_u^2-d_v^2|}{2\Delta^2} \\
	&\geq \frac{1}{\Delta^2}\left( \frac{1}{2}\sum_{uv\in E(G)} |d_u^2-d_v^2| \right)= \frac{SO_1(G)}{\Delta^2} .
	\end{align*}
	Similarly, we have the result for the upper bound.
	\qed
\end{proof}

The following proposition gives an upper bound for $SO_2(G-e)$:

\begin{theorem}\label{thm(G-e)-so2}
	Let $G=(V,E)$ be a non-regular graph and $e=uv\in E$. Also let  $\Delta$ and $\delta$, be the  maximum degree and minimum degree of vertices  in $G$, respectively. Then,
	$$SO_2(G) > \frac{\delta^2}{\Delta^2} \left( SO_2(G-e) + \frac{1}{2}-\frac{\Delta^2}{2\delta^2}\right). $$
\end{theorem}

\begin{proof}
	By using Propositions \ref{pro(G-e)} and \ref{pro-so1-so2}, we have
	\begin{align*}
	SO_2(G) &\geq \frac{SO_1(G)}{\Delta^2} \\
	&> \frac{1}{\Delta^2}\left( SO_1(G-e) + \frac{1}{2}(\delta^2-\Delta^2) \right)\\
	&\geq \frac{1}{\Delta^2}\left( \delta^2 SO_2(G-e) + \frac{1}{2}(\delta^2-\Delta^2) \right),
	\end{align*}
	and therefore we have the result.
	\qed
\end{proof}

The  following proposition  gives a lower and upper bounds for $SO_3(G)$ based $SO_1(G)$, $\delta$ and $\Delta$.

\begin{proposition}\label{pro-so1-so3}
	Let $G=(V,E)$ be a graph with $|E|=m$ and minimum degree $\delta$ and maximum degree $\Delta$. Then
	$$\sqrt{2}\pi \left(\frac{SO_1(G)+m\delta^2}{\Delta} \right) \leq SO_3(G)\leq \sqrt{2}\pi \left(\frac{SO_1(G)+m\Delta^2}{\delta} \right).$$
\end{proposition}

\begin{proof}
	First we obtain the upper bound. We have
	\begin{align*}
	SO_3(G) &= \sum_{uv\in E(G)}\sqrt{2}\frac{d_u^2+d_v^2}{d_u+d_v}\pi \\
	&\leq \sqrt{2}\pi \left( \sum_{uv\in E(G)} \frac{|d_u^2-d_v^2|+2\Delta^2}{2\delta} \right) \\
	&\leq \sqrt{2}\pi \left(\frac{1}{\delta} \left( \frac{1}{2} \sum_{uv\in E(G)} |d_u^2-d_v^2| \right)+ \frac{m\Delta^2}{\delta}\right) \\
	&= \sqrt{2}\pi \left(\frac{SO_1(G)+m\Delta^2}{\delta} \right).
	\end{align*}
	Similarly, we have the result for the lower bound.
	\qed
\end{proof}

By similar argument as Theorem \ref{thm(G-e)-so2}, by using Propositions \ref{pro(G-e)} and \ref{pro-so1-so3}, we have

\begin{theorem}\label{thm(G-e)-so3}
	Let $G=(V,E)$ be a  graph and $e=uv\in E$ and $|E|=m$. Also let  $\Delta$ and $\delta$, be the  maximum degree and minimum degree of vertices  in $G$, respectively. Then,
	$$SO_3(G) > \frac{\delta}{\Delta} \left( SO_3(G-e)\right) + \frac{(2m+1)\pi}{\sqrt{2}\Delta}\left(\delta^2 - \Delta^2 \right). $$
\end{theorem}

\begin{proposition}\label{pro-so1-so4}
	Let $G=(V,E)$ be a graph with $|E|=m$ and minimum degree $\delta$ and maximum degree $\Delta$. Then
	$$\frac{\pi\delta^2}{2\Delta^2}\left( m\delta^2+SO_1(G) \right)  \leq SO_4(G)\leq \frac{\pi\Delta^2}{2\delta^2}\left( m\Delta^2+SO_1(G) \right) .$$
\end{proposition}

\begin{proof}
	First we obtain the upper bound. We have
	\begin{align*}
	SO_4(G) &= \frac{1}{2}\sum_{uv\in E(G)}\left(\frac{d_u^2+d_v^2}{d_u+d_v}\right)^2\pi \\
	&=  \frac{\pi}{2}\sum_{uv\in E(G)}\left(\frac{(d_u^2+d_v^2)(d_u^2+d_v^2)}{(d_u+d_v)^2}\right)  \\
	&\leq \frac{\pi}{2}\sum_{uv\in E(G)}\left(\frac{(d_u^2+d_v^2)(2\Delta^2)}{(2\delta)^2}\right) \\
	&\leq \frac{\pi \Delta^2}{2\delta^2}\left(\frac{1}{2}\sum_{uv\in E(G)} 2\Delta^2 + |d_u^2-d_v^2| \right)\\
	&= \frac{\pi\Delta^2}{2\delta^2}\left( m\Delta^2+SO_1(G) \right).
	\end{align*}
	Similarly, we have the result for the lower bound.
	\qed
\end{proof}

By similar argument as Theorem \ref{thm(G-e)-so2}, by using Propositions \ref{pro(G-e)} and \ref{pro-so1-so3}, we have

\begin{theorem}\label{thm(G-e)-so4}
	Let $G=(V,E)$ be a graph and $e=uv\in E$ and $|E|=m$. Also let  $\Delta$ and $\delta$, be the  maximum degree and minimum degree of vertices  in $G$, respectively. Then,
	$$SO_4(G) >  SO_4(G-e)+ \frac{(2m+1)\pi\delta^2(\delta^2-\Delta^2)}{2\Delta^2}. $$
\end{theorem}

By similar argument as Proposition \ref{pro-so1-so2}, we have:

\begin{proposition}\label{pro-so1-so5}
	Let $G$ be a graph with minimum degree $\delta$ and maximum degree $\Delta$. Then
	$$\frac{2\sqrt{2}\pi SO_1(G)}{2\Delta+1} \leq SO_5(G)\leq \frac{2\sqrt{2}\pi SO_1(G)}{2\delta+1}.$$
\end{proposition}

By similar argument as Theorem \ref{thm(G-e)-so2}, by using Propositions \ref{pro(G-e)} and \ref{pro-so1-so5}, we have

\begin{theorem}\label{thm(G-e)-so5}
	Let $G=(V,E)$ be a non-regular graph and $e=uv\in E$. Also let  $\Delta$ and $\delta$, be the  maximum degree and minimum degree of vertices  in $G$, respectively. Then,
	$$SO_5(G) >  SO_5(G-e) + \frac{\sqrt{2}\pi}{2\Delta+1}\left(\delta^2 - \Delta^2 \right). $$
\end{theorem}

\begin{proposition}\label{pro-so1-so6}
	Let $G=(V,E)$ be a graph with $|E|=m$ and minimum degree $\delta$ and maximum degree $\Delta$. Then
	$$ SO_6(G)\leq \frac{2\pi(\Delta^2-\delta^2)SO_1(G)}{(\sqrt{2}+2\delta\sqrt{2})^2}.$$
\end{proposition}

\begin{proof}
 We have
	\begin{align*}
	SO_6(G) &=  \pi\sum_{uv\in E(G)}\left(\frac{d_u^2-d_v^2}{\sqrt{2}+2\sqrt{d_u^2+d_v^2}}\right)^2 \\
	&\leq  \pi\sum_{uv\in E(G)}\frac{|d_u^2-d_v^2|(\Delta^2-\delta^2)}{(\sqrt{2}+2\delta\sqrt{2})^2} \\
	&=\frac{2\pi(\Delta^2-\delta^2)}{(\sqrt{2}+2\delta\sqrt{2})^2}\left( \frac{1}{2}\sum_{uv\in E(G)}|d_u^2-d_v^2|\right)  \\
	&=\frac{2\pi(\Delta^2-\delta^2)SO_1(G)}{(\sqrt{2}+2\delta\sqrt{2})^2} . 	
	\end{align*}
	\qed
\end{proof}

\section{Acknowledgements} 

	The  first author would like to thank the Research Council of Norway and Department of Informatics, University of Bergen for their support.

\end{document}